\documentclass[a4paper,11pt]{amsart}
\usepackage[utf8]{inputenc}

\usepackage{amsmath,amssymb,amsthm}
\usepackage{hyperref,url}

\usepackage{graphicx,color}
\graphicspath{{Fig/}}

\usepackage[matrix,arrow]{xy}

\newtheorem{thm}{Theorem}[section]

\newtheorem{lem}[thm]{Lemma}
\newtheorem{cor}[thm]{Corollary}

\theoremstyle{definition}
\newtheorem{dfn}[thm]{Definition}
\newtheorem{rem}[thm]{Remark}
\newtheorem{exl}[thm]{Example}

\renewcommand{\phi}{\varphi}

\def\R{\mathbb R}

\def\Sph{\mathbb S}

\def\area{\mathrm{area}}
\def\vol{\mathrm{vol}}
\def\dist{\operatorname{dist}}
\def\scal{\mathrm{scal}}
\def\Ric{\mathrm{Ric}}
\def\pr{\operatorname{pr}}

\title{Discrete curvature}

\author{Ivan Izmestiev}
\address{
TU Wien,
Institute of Discrete Mathematics and Geometry,
Wiedner Hauptstra{\ss}e 8-10/104,
1040 Vienna, Austria}
\email{izmestiev@dmg.tuwien.ac.at}

\begin{document}

\begin{abstract}
The combination of words ``discrete curvature'' is only an apparent contradiction.
In this survey article we describe curvature notions associated with polygons, polyhedral surfaces, and with abstract polyhedral manifolds.
Several theorems about the discrete curvature are stated that repeat literally classical theorems of differential and Riemannian geometry: Theorema Egregium, Gauss--Bonnet theorem, and the Chern--Gauss--Bonnet theorem among the others.
Some convergence results are also mentioned: under certain assumptions the discrete curvature tends to the smooth curvature as a smooth object is approximated by polyhedral ones.
\end{abstract}

\maketitle


\section{Curvature of curves and polygons}
\subsection{The total curvature of planar curves}
The \emph{curvature} of a smooth curve in the plane is the rotation speed of its tangent relative to the arc length.
Let $\gamma \colon [a,b] \to \R^2$ be an arc-length parametrized curve: $\|\dot\gamma(t)\| = 1$ for all $t \in [a,b]$.
Then the acceleration vector $\ddot\gamma$ is orthogonal to~$\dot\gamma$ because~of
\[
\langle \dot\gamma, \ddot\gamma \rangle = \frac12 \frac{d}{dt} \langle \dot\gamma, \dot\gamma \rangle = 0.
\]
Thus the norm of $\ddot\gamma$ is equal to the rotation speed of the tangent, and one has the following formula for the curvature:
\[
\kappa := \|\ddot\gamma\|.
\]
Let us equip the curvature with a sign: plus if the curve turns left, and minus if the curve turns right.
Formally, choose a unit normal $\nu$ at every point of the curve so that the orthonormal basis $(\dot\gamma, \nu)$ is positively oriented.
Then the \emph{signed curvature} $\kappa^s$ of $\gamma$ is obtained from
\[
\ddot\gamma = \kappa^s \nu.
\]

Now consider a polygon $C_1 \ldots C_n$ in the plane.
A point moving along a polygonal line changes the direction of its motion  abruptly at the vertices.
Let us assume that the direction is never changed to its opposite.
Then at each vertex one has the turning angle
\[
\kappa^s_i = \angle(\overrightarrow{C_{i-1}C_i}, \overrightarrow{C_iC_{i+1}}) \in (-\pi, \pi),
\]
and we call this angle the \emph{signed curvature} of the polygon at the vertex $C_i$.
Figure~\ref{fig:SignedCurv} illustrates the definitions of the signed curvature in the smooth and the discrete situation.

\begin{figure}[ht]
\begin{picture}(0,0)%
\includegraphics{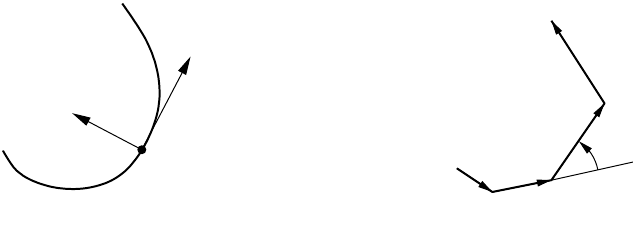}%
\end{picture}%
\setlength{\unitlength}{4144sp}%
\begingroup\makeatletter\ifx\SetFigFont\undefined%
\gdef\SetFigFont#1#2#3#4#5{%
  \reset@font\fontsize{#1}{#2pt}%
  \fontfamily{#3}\fontseries{#4}\fontshape{#5}%
  \selectfont}%
\fi\endgroup%
\begin{picture}(4834,1725)(-2805,-1318)
\put(1730,-757){\makebox(0,0)[lb]{\smash{{\SetFigFont{10}{12.0}{\sfdefault}{\mddefault}{\updefault}{\color[rgb]{0,0,0}$\kappa^s_i$}%
}}}}
\put(1426,-1117){\makebox(0,0)[lb]{\smash{{\SetFigFont{10}{12.0}{\rmdefault}{\mddefault}{\updefault}{\color[rgb]{0,0,0}$C_i$}%
}}}}
\put(1869,-331){\makebox(0,0)[lb]{\smash{{\SetFigFont{10}{12.0}{\rmdefault}{\mddefault}{\updefault}{\color[rgb]{0,0,0}$C_{i+1}$}%
}}}}
\put(831,-1254){\makebox(0,0)[lb]{\smash{{\SetFigFont{10}{12.0}{\rmdefault}{\mddefault}{\updefault}{\color[rgb]{0,0,0}$C_{i-1}$}%
}}}}
\put(-2481,-393){\makebox(0,0)[lb]{\smash{{\SetFigFont{10}{12.0}{\sfdefault}{\mddefault}{\updefault}{\color[rgb]{0,0,0}$\ddot\gamma = \kappa^s \nu$}%
}}}}
\put(-1304,-196){\makebox(0,0)[lb]{\smash{{\SetFigFont{10}{12.0}{\sfdefault}{\mddefault}{\updefault}{\color[rgb]{0,0,0}$\dot\gamma$}%
}}}}
\end{picture}%
\caption{The signed curvature of curves and polygons.}
\label{fig:SignedCurv}
\end{figure}
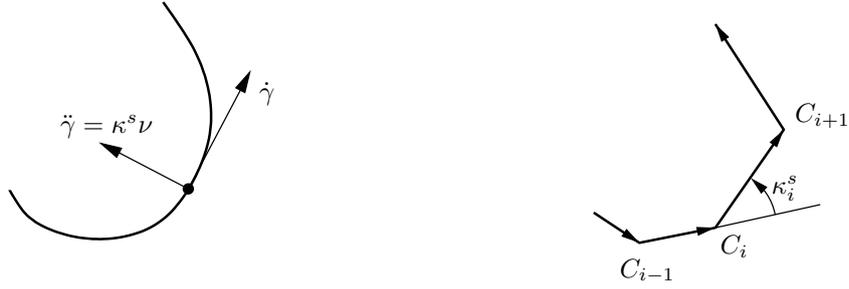

\begin{thm}
The total signed curvature of a closed plane curve is an integer multiple of $2\pi$:
\[
\int_a^b \kappa_s\, dt = 2\pi k, \quad \text{respectively} \quad \sum_{i=1}^n \kappa_i^s = 2\pi k.
\]
\end{thm}
\begin{proof}
For a smooth curve parametrized with a unit speed there is a continuous function $\alpha \colon [a,b] \to \R$ such that $\dot\gamma(t) = (\cos\alpha(t), \sin\alpha(t))$.
One has
\[
\ddot \gamma = \dot\alpha (-\sin\alpha, \cos\alpha) = \dot\alpha \nu,
\]
which implies $\kappa^s = \dot\alpha$.
Therefore
\[
\int_a^b \kappa^s\, dt = \int_a^b \dot\alpha\, dt = \alpha(b) - \alpha(a),
\]
which is an integer multiple of $2\pi$ because $\dot\gamma(b) = \dot\gamma(a)$.

In the discrete case the sum of the signed curvatures is the total turning angle of the directed edges of the polygon as one makes a closed tour.
The direction of the motion is at the end the same as in the beginning, but the angle which describes the direction is well-defined only modulo $2\pi$.
Hence the total signed curvature is an integer multiple of $2\pi$.
\end{proof}

If the total curvature is equal to $2\pi k$, then the number $k$ is called the \emph{turning number} or the rotation index of the curve.
Figure \ref{fig:TurningNumber} shows two examples.

\begin{figure}[ht]
\includegraphics[width=.8\textwidth]{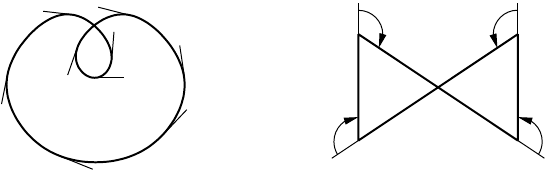}
\caption{A smooth curve with the turning number $\pm 2$ and a polygon with the turning number $0$.}
\label{fig:TurningNumber}
\end{figure}

The turning number of a simple (that is, non-self-intersecting) closed curve is $\pm 1$, where the sign depends on the orientation of the curve.
Intuitively this is clear; a proof in the smooth case was given by Hopf \cite{Hopf1935}, see also \cite{Hsiung1997}.

\subsection{The total curvature of spatial curves}
For an oriented plane curve it was possible to equip the curvature with a sign according to whether the velocity vector turns left or right as a point moves along the curve.
For a space curve there is no left and right, therefore we define the curvature as a non-negative quantity in both smooth and the discrete setting:
\begin{align*}
\kappa &:= \|\ddot \gamma\|\ \text{ for a unit-speed curve,}\\
\kappa_i &:= \pi - \angle C_{i-1}C_iC_{i+1}\ \text{ for a polygon}.
\end{align*}

\begin{thm}[Fenchel \cite{Fen29}]
\label{thm:Fenchel}
The total curvature of a closed space curve is bigger or equal $2\pi$:
\[
\int_a^b \kappa\, dt \ge 2\pi, \quad \text{respectively} \quad \sum_{i=1}^n \kappa_i \ge 2\pi.
\]
The equality takes place if and only if the curve is planar and convex.
\end{thm}
For the proof we will need a definition and several lemmas.

\begin{dfn}
The \emph{tangent indicatrix} of a smooth space curve $\gamma \colon [a,b] \to \R^3$ is a curve in the unit sphere $\Sph^2 = \{x \in \R^3 \mid \|x\|=1\}$ given by
\[
T_\gamma \colon [a,b] \to \Sph^2, \quad T_\gamma(t) = \frac{\dot\gamma(t)}{\|\dot\gamma(t)\|}.
\]
The \emph{tangent indicatrix} of a space polygon $C_1 \ldots C_n$ is the polygon in $\Sph^2$ obtained by connecting consecutive points of the sequence
\[
T_i = \frac{C_{i+1} - C_i}{\|C_{i+1} - C_i\|}, \quad i = 1, \ldots, n-1
\]
by arcs of great circles.
If the polygon is closed, one considers the indices modulo $n$, introduces the point $T_n$ and connects it with $T_1$ and $T_{n-1}$.
\end{dfn}

\begin{figure}[ht]
\begin{center}
\begin{picture}(0,0)%
\includegraphics{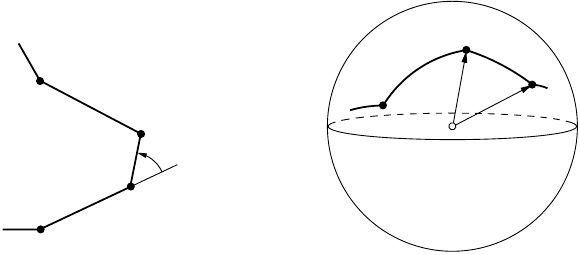}%
\end{picture}%
\setlength{\unitlength}{3729sp}%
\begingroup\makeatletter\ifx\SetFigFont\undefined%
\gdef\SetFigFont#1#2#3#4#5{%
  \reset@font\fontsize{#1}{#2pt}%
  \fontfamily{#3}\fontseries{#4}\fontshape{#5}%
  \selectfont}%
\fi\endgroup%
\begin{picture}(4896,2147)(3219,-1845)
\put(3589,-1781){\makebox(0,0)[lb]{\smash{{\SetFigFont{9}{10.8}{\sfdefault}{\mddefault}{\updefault}{\color[rgb]{0,0,0}$C_{i-1}$}%
}}}}
\put(4356,-1425){\makebox(0,0)[lb]{\smash{{\SetFigFont{9}{10.8}{\sfdefault}{\mddefault}{\updefault}{\color[rgb]{0,0,0}$C_i$}%
}}}}
\put(4419,-735){\makebox(0,0)[lb]{\smash{{\SetFigFont{9}{10.8}{\sfdefault}{\mddefault}{\updefault}{\color[rgb]{0,0,0}$C_{i+1}$}%
}}}}
\put(4512,-1008){\makebox(0,0)[lb]{\smash{{\SetFigFont{9}{10.8}{\sfdefault}{\mddefault}{\updefault}{\color[rgb]{0,0,0}$\kappa_i$}%
}}}}
\put(7724,-610){\makebox(0,0)[lb]{\smash{{\SetFigFont{9}{10.8}{\sfdefault}{\mddefault}{\updefault}{\color[rgb]{0,0,0}$T_{i-1}$}%
}}}}
\put(7439,-214){\makebox(0,0)[lb]{\smash{{\SetFigFont{9}{10.8}{\sfdefault}{\mddefault}{\updefault}{\color[rgb]{0,0,0}$\kappa_i$}%
}}}}
\put(6950,-44){\makebox(0,0)[lb]{\smash{{\SetFigFont{9}{10.8}{\sfdefault}{\mddefault}{\updefault}{\color[rgb]{0,0,0}$T_i$}%
}}}}
\end{picture}%
\end{center}
\caption{The tangent indicatrix of a space polygon.}
\label{fig:DiscrTangIndic}
\end{figure}
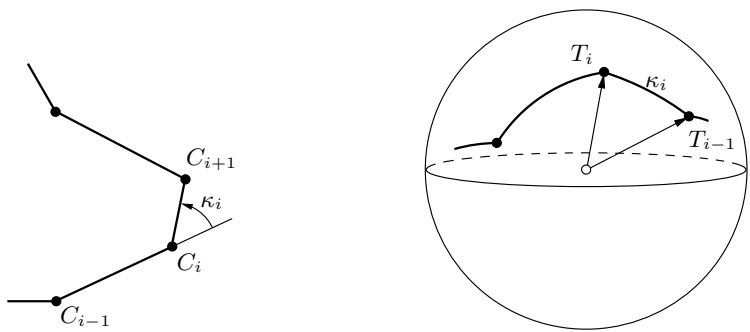

\begin{lem}
The total curvature of a space curve is equal to the length of its tangent indicatrix.
\end{lem}
\begin{proof}
In the discrete case this is true because the length of the arc $T_{i-1}T_i$ is equal to the turning angle $\kappa_i$, see Figure \ref{fig:DiscrTangIndic}.

In the smooth case assume without loss of generality that $\gamma$ is arc-length parametrized.
Then one has
\[
T_\gamma = \dot\gamma \ \Rightarrow \ \dot T_\gamma = \ddot \gamma.
\]
Hence the length of $T_\gamma$ is the integral of $\|\ddot\gamma\| = \kappa$, and the lemma is proved.
\end{proof}

\begin{lem}
\label{lem:TangIndLarge}
The tangent indicatrix of a closed space curve is not contained in any open hemisphere.
\end{lem}
\begin{proof}
Let us prove this in the discrete case.
Assume the contrary: there is $v \in \R^3$ such that $\langle T_i, v \rangle > 0$ for all~$i$.
Then due to $\sum_{i=1}^n (C_{i+1} - C_i) = 0$ one has
\[
0 = \sum_{i=1}^n \langle C_{i+1} - C_i, v \rangle = \sum_{i=1}^n \|C_{i+1} - C_i\| \langle T_i, v \rangle > 0,
\]
which is a contradiction.
The proof in the smooth case is similar.
\end{proof}

Lemma \ref{lem:TangIndLarge} implies that almost every great circle intersects the tangent indicatrix at least twice.

\begin{lem}[Crofton]
\label{lem:CroftonSphere}
The length of a spherical curve is $\pi$ times the average number of its intersections with great circles.
\end{lem}
The average is taken with respect to the following measure on the space of great circles.
A great circle is the intersection of $\Sph^2$ with a $2$-dimensional linear subspace of $\R^3$; the orthogonal complement defines a bijection between $2$-dimensional linear subspaces and $1$-dimensional linear subspaces; the latter form the projective plane which is doubly covered by the sphere.
Thus the standard measure on the sphere can be pushed to a measure on the space of great circles.
In other words, the lemma states that for any curve $T \subset \Sph^2$ one has
\[
\operatorname{Length}(T) = \frac{1}{4} \int_{\Sph^2} |T \cap v^\perp|\, dv.
\]
(The total measure of $\Sph^2$ is $4\pi$.)

\begin{proof}
Consider the discrete case.
The number of intersections of a spherical polygon with a great circle is equal to the sum of the numbers of intersections of its edges with this circle (for almost all circles); the length of a polygon is the sum of the lengths of its edges.
Therefore it suffices to prove the lemma for a single edge of the polygon, that is for any arc of a great circle.

Since almost every great circle intersects an arc either once or not at all, one has to show that the length of a spherical arc is $\pi$ times the probability that a random great circle intersects this arc.
A great circle $v^\circ := v^\perp \cap \Sph^2$ intersects an arc $pq$ if and only if the inner products $\langle v, p \rangle$ and $\langle v, q \rangle$ have opposite signs.
This in turn is equivalent to
\[
v \in (H_p^+ \cap H_q^-) \cup (H_p^- \cap H_q^+),
\]
where $H_p^+ = \{x \in \Sph^2 \mid \langle x, p \rangle \ge 0\}$ etc.
It is not hard to see that $(H_p^+ \cap H_q^-) \cup (H_p^- \cap H_q^+)$ is the union of two spherical lunes of angle equal to the length of $pq$.
Thus the probability of $v^\circ$ intersecting $pq$ is equal to the probability of $v$ lying in those lunes, which is $\operatorname{Length}(pq)/\pi$.
See Figure \ref{fig:CroftonSphere}.
This proves the lemma in the discrete case.

For smooth curves the proof is longer and can be found in \cite[Theorem 3.1.9]{Hsiung1997}.
\end{proof}

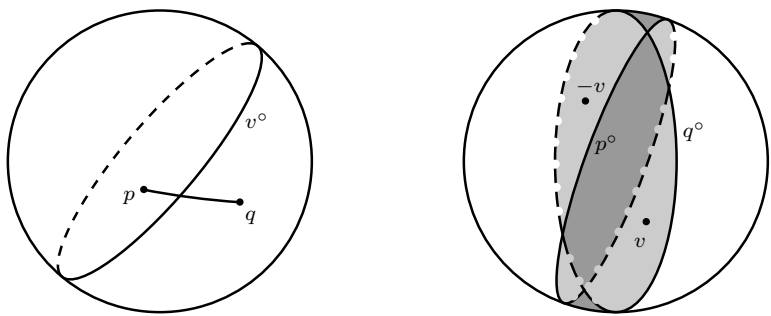
\begin{figure}[ht]
\begin{center}
\begin{picture}(0,0)%
\includegraphics{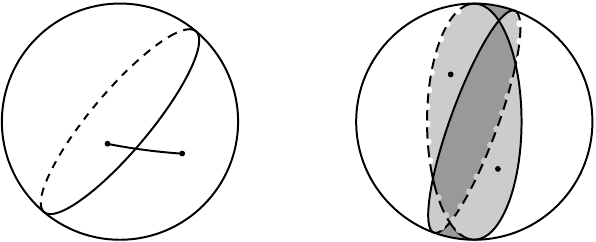}%
\end{picture}%
\setlength{\unitlength}{4144sp}%
\begingroup\makeatletter\ifx\SetFigFont\undefined%
\gdef\SetFigFont#1#2#3#4#5{%
  \reset@font\fontsize{#1}{#2pt}%
  \fontfamily{#3}\fontseries{#4}\fontshape{#5}%
  \selectfont}%
\fi\endgroup%
\begin{picture}(4530,1828)(-14,-975)
\put(3997, 82){\makebox(0,0)[lb]{\smash{{\SetFigFont{8}{9.6}{\rmdefault}{\mddefault}{\updefault}{\color[rgb]{0,0,0}$q^\circ$}%
}}}}
\put(684,-289){\makebox(0,0)[lb]{\smash{{\SetFigFont{8}{9.6}{\rmdefault}{\mddefault}{\updefault}{\color[rgb]{0,0,0}$p$}%
}}}}
\put(1406,137){\makebox(0,0)[lb]{\smash{{\SetFigFont{8}{9.6}{\rmdefault}{\mddefault}{\updefault}{\color[rgb]{0,0,0}$v^\circ$}%
}}}}
\put(1406,-408){\makebox(0,0)[lb]{\smash{{\SetFigFont{8}{9.6}{\rmdefault}{\mddefault}{\updefault}{\color[rgb]{0,0,0}$q$}%
}}}}
\put(3708,-570){\makebox(0,0)[lb]{\smash{{\SetFigFont{8}{9.6}{\rmdefault}{\mddefault}{\updefault}{\color[rgb]{0,0,0}$v$}%
}}}}
\put(3366,354){\makebox(0,0)[lb]{\smash{{\SetFigFont{8}{9.6}{\rmdefault}{\mddefault}{\updefault}{\color[rgb]{0,0,0}$-v$}%
}}}}
\put(3477,  6){\makebox(0,0)[lb]{\smash{{\SetFigFont{8}{9.6}{\rmdefault}{\mddefault}{\updefault}{\color[rgb]{0,0,0}$p^\circ$}%
}}}}
\end{picture}%

\end{center}
\caption{Determining the probability that a random great circle intersects a given arc.}
\label{fig:CroftonSphere}
\end{figure}

\subsection{The generalized total curvature}
Let $\gamma \colon [a,b] \to \R^n$ be an arbitrary curve, that is just any continuous map.
One defines the total curvature in a way similar to the definition of the length, by using inscribed polygons:
\[
\operatorname{Curv}(\gamma) = \sup_C \operatorname{Curv}(C),
\]
where the supremum is taken over all polygons inscribed in $\gamma$:
\[
C_i = \gamma(t_i), \quad a=t_1 < t_2 < \cdots < t_n = b,
\]
and $\operatorname{Curv}(C)$ is the sum of the turning angles of $C$.
If the supremum is finite, than one says that the curve has \emph{bounded total curvature}.
Milnor \cite{Milnor1950} has shown that for smooth curves the above definition agrees with the traditional one:
\[
\operatorname{Curv}(\gamma) = \int_a^b \kappa\, dt.
\]
Theorem \ref{thm:Fenchel} then clearly holds for all curves of bounded total curvature.
Besides, one can derive its smooth case from the discrete case by approximation.

For a survey on curves of bounded total curvature see \cite{Sullivan2008}.

%
%

\section{Curvature of surfaces}
\subsection{Surfaces}
The curvature of a curve is the rotation speed of its tangent line.
Similarly, the curvature of a surface tells how does the tangent plane rotate as the point of tangency moves in different directions along the surface.
Since the tangent plane is orthogonal to the unit normal $\nu$, the curvature is encoded in the \emph{shape operator}:
\begin{equation}
\label{eqn:Shape}
S \colon T_p M \to T_p M, \quad S(X) = -D_X \nu.
\end{equation}
This operator is self-adjoint, and therefore it can be orthogonally diagonalized.
Its eigenvalues $\kappa_1, \kappa_2$ are called the principal curvatures of the surface.
One defines the Gaussian curvature and the mean curvature as
\[
K = \kappa_1\kappa_2, \quad H = \frac{\kappa_1 + \kappa_2}2,
\]
in other words, the determinant and half of the trace of the shape operator.

Tangent planes and normals to a polyhedral surface are well-defined only in the interiors of the faces.
The normal changes at the edges and at the vertices, so it is there that we should look for the curvature.
As in the case of curves, the discrete curvature should be related to the angles between the faces of a polyhedral surface.
Although we will not find any analog of the shape operator, we will find discrete versions of the Gaussian and the mean curvature.

\subsection{Steiner formula and Minkowski theory}
Curvatures of polyhedral surfaces were discovered by Steiner \cite{Steiner1840a}.
For a convex closed surface $M \subset \R^3$ bounding a convex body $N$ he considered its outer parallel surface
\[
M_r := \{x \in \R^3\subset N \mid \dist(x,M) = r\}
\]
and computed the area of $M_r$ and the volume of the body $N_r$ enclosed by~$M_r$.

\begin{thm}
\label{thm:Steiner}
If $M$ is a convex closed smooth surface, then
\begin{gather*}
\area(M_r) = \area(M) + 2r \int_M H\, d\area + r^2 \int_M K\, d\area,\\
\vol(N_r) = \vol(N) + r \cdot \area(M)+ r^2  \int_M H\, d\area + \frac{r^3}3 \int_M K\, d\area.
\end{gather*}
If $M$ is a convex closed polyhedral surface, then
\begin{gather*}
\area(M_r) = \area(M) + r \sum_e \beta_e \ell_e + r^2 \sum_v \beta_v,\\
\vol(N_r) = \vol(N) + r \cdot \area(M) + \frac{r^2}2 \sum_e \beta_e \ell_e + \frac{r^3}3 \sum_v \beta_v.
\end{gather*}
Here the sums are taken over all edges $e$, respectively over all vertices $v$.
By $\ell_e$ we denote the edge length and by $\beta_e$ the external angle at $e$ (the angle between the normals to the faces adjacent to $e$).
By $\beta_v$ we denote the external angle at a vertex $v$ (the area of the intersection of the cone spanned by the normals around $v$ with the unit sphere centered at $v$).
\end{thm}
\begin{proof}
In the smooth case there is a diffeomorphism
\[
f \colon M \to M_r, \quad f(p) = p + r\nu,
\]
where $\nu$ is the field of outward unit normals to $M$.
The tangent plane to $M_r$ at $f(p)$ is parallel to the tangent plane to $M$ at $p$ (whence the name ``parallel surface''), and the principal radii of curvature of $M_r$ at $f(p)$ are greater by $r$ than those of $M$ at $p$.
This implies that the Jacobian determinant of $f$ is equal to $(1+r\kappa_1)(1+r\kappa_2)$.
Integration of the determinant yields the formula for the area.
The volume element in $N_r \setminus N$ is equal to $dr$ times the area element of $M_x$, which implies $\vol(N_r) = \vol(N) + \int_0^r \area(M_x)\, dx$, and hence the formula for the volume.

\begin{figure}[ht]
\begin{center}
\includegraphics[width=.5\textwidth]{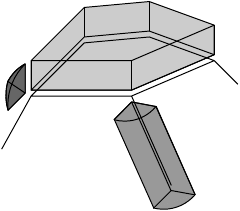}
\end{center}
\caption{Decomposition of the $r$-neighborhood of a convex polyhedron.}
\label{fig:SteinerPoly}
\end{figure}

In the discrete case both formulas follow from decomposition of $M_r$ and $N_r \setminus N$ into parts whose nearest neighbor in $M$ belongs to a face, to an edge or is a vertex.
For $N$ these parts are: prisms over the faces of $M$, wedges at the edges of $M$, ball sectors at the vertices of $M$, see Figure \ref{fig:SteinerPoly}.
\end{proof}


Half a century after Steiner's observation, Minkowski has deepened the analogy between the formulas in Theorem \ref{thm:Steiner}.
The space $\mathcal{K}$ of convex bodies in $\R^3$ is a metric space with respect to the Hausdorff distance
\[
\dist(N,P) := \min\{r \mid N \subset P_r, P \subset N_r\}.
\]
The volume (Jordan measure) of a convex body is a continuous function on $\mathcal{K}$.
Both the set of bodies with smooth boundary and the set of convex polyhedra are dense within $\mathcal{K}$.
From this one can conclude that the volume of the $r$-neighborhood of an arbitrary convex body is a polynomial of degree $3$ in $r$:
\begin{equation}
\label{eqn:Minkowski}
\vol(N_r) = \vol(N) + V_1(N) r + V_2(N) r^2 + V_3(N) r^3,
\end{equation}
and that the coefficients of this polynomial depend continuously on $N$.

Steiner and Minkowski formulas can be extended to higher dimensions in a straightforward way, see Section \ref{sec:HighDim}.
A further generalization of the Steiner formula, an expansion of $\vol(\lambda N + \mu P)$, led Minkowski to the notion of mixed volumes.
For a short survey of this theory see \cite{Izmestiev2022}.
For more breadth and depth see the books \cite{Bonnesen1987, Schneider2014}.

The coefficient $V_1(N) = \left.\frac{d}{dr}\right|_{r=0} \vol(N_r)$ in \eqref{eqn:Minkowski} can be used as a definition of the surface area of an arbitrary convex body.
Let us now look at the other two coefficients.

\subsection{The total mean curvature and integral geometry}
As we proved in Theorem \ref{thm:Steiner}, the coefficient at $r^2$ in the polynomial \eqref{eqn:Minkowski} has the following values in special cases:
\[
V_2(N) =
\begin{cases}
\int_{M} H\, d\area, &\text{if }M = \partial N \text{ is smooth},\\
\frac12 \sum_e \beta_e \ell_e, &\text{if } N\text{ is a polyhedron}.
\end{cases}
\]
Due to its meaning in the smooth case, $V_2$ is usually called the \emph{total mean curvature} of the convex surface $M$, even if this surface is not smooth.
The continuity of the function $V_2$ implies that, when a convex smooth surface is approximated (in the Hausdorff metric) by polyhedral ones, then $\frac12 \sum_e \beta_e \ell_e$ tends to $\int_{\partial K} H\, d\area$, and vice versa.

There is a very nice integral-geometric description of the total mean curvature.
For this, recall the spherical Crofton formula \ref{lem:CroftonSphere}.
This formula has a Euclidean analog, which is as follows.
\begin{thm}
\label{thm:CroftonEuc}
The length of a closed convex curve in the plane is $\pi$ times its average width, where the width is the length of the orthogonal projection to a line:
\[
\operatorname{Length}(\gamma) = \frac12 \int_{\Sph^1} \operatorname{Length}(\pr_v (\gamma))\, dv.
\]
\end{thm}

The Crofton formula can be generalized to surfaces in two different ways, the first of them known as the Cauchy formula.

\begin{thm}
\label{thm:Cauchy}
The area of a closed convex surface is twice its average projection area on a random plane:
\[
\operatorname{Area}(M) = \frac{1}{2\pi} \int_{\Sph^2} \operatorname{Area}(\pr_{v^\perp}(M))\, dv.
\]
\end{thm}
Both theorems \ref{thm:CroftonEuc} and \ref{thm:Cauchy} can be proved in the polyhedral case by observing that the projection of a polygon/polyhedron is covered by the projections of its edges/faces twice.
Therefore it suffices to compute the average projection length/area of a segment/polygon, which boils down to integrating the cosine of the slope.

The second generalization of the Crofton formula results in a formula for the total mean curvature.
\begin{thm}
The total mean curvature of a closed convex surface is $2\pi$ times its average width:
\[
a_2 = \frac12 \int_{\Sph^2} \operatorname{Length}(\pr_v M),\ dv.
\]
\end{thm}
Here a direct proof in the discrete case is also possible but more complicated than for the two theorems above.
The proof in the smooth case is nice and follows from the formula $\vol(N) = \frac13 \int_M h\, d\area$, where $h(p) = \langle p, \nu \rangle$.
Whichever case you prove, it implies the formula for general closed convex surfaces due to the Hausdorff continuity of both sides of the formula.

\subsection{The total Gaussian curvature and the extrinsic Gauss-Bonnet theorem}
Let us look at the cubic term in equation \eqref{eqn:Minkowski}.
Asymptotically, the volume of $N_r$ grows as the volume of the ball of radius $r$, which implies
\[
V_3(N) = \frac{4\pi}3 \text{ for every convex body } N.
\]
Comparing this with the cubic terms in the Steiner formula for smooth and polyhedral surfaces, one obtains the following theorem.

\begin{thm}
The total Gaussian curvature of any closed convex smooth surface and the total vertex curvature of any closed convex polyhedral surface is equal to $4\pi$:
\[
\int_M K\, d\area = 4\pi, \quad \sum_v \beta_v = 4\pi.
\]
\end{thm}

The following proof does not use the Steiner formula.
\begin{proof}
Translate the normal cones of the vertices of a convex polyhedron to a common apex.
They will cover $\R^3$ without overlaps.
This implies the theorem in the discrete case.

In the smooth case recall that the Gaussian curvature is the determinant of the shape operator, which is equal to the determinant of the Gauss map $\Gamma \colon M \to \Sph^2$.
If $K > 0$, then the Gauss map is a diffeomorphism.
Use this diffeomorphism as a variable change in the integral:
\begin{equation}
\label{eqn:HopfProof}
\int_{M} K \, d\area = \int_{M} \det d\Gamma\, d\area_M = \int_{\Sph^2} d\area = \area(\Sph^2).
\end{equation}
\end{proof}

The above argument was generalized to arbitrary closed smooth surfaces by Hopf \cite{Hopf1926}.

\begin{thm}[Gauss-Bonnet-Hopf]
\label{thm:GBH}
The total Gaussian curvature of any closed smooth surface is equal to $2\pi$ times its Euler characteristic:
\[
\int_M K\, d\area = 2\pi \chi(M).
\]
\end{thm}
\begin{proof}
The Gauss map covers a generic point $\nu \in \Sph^2$ with a constant multiplicity, called the degree:
\[
\deg \Gamma := \sum_{p \in \Gamma^{-1}(\nu)} \operatorname{sgn} \det d\Gamma(p).
\]
The number $\deg \Gamma$ appears as a factor on the right hand side of \eqref{eqn:HopfProof}.
It can be shown that the degree of the Gauss map is equal to half the Euler characteristic of the surface:
\[
\int_M K\, d\area = \deg \Gamma \cdot \area(\Sph^2) = \frac{\chi(M)}2 \cdot 4\pi = 2\pi \chi(M).
\]
\end{proof}

The discrete analog of Theorem \ref{thm:GBH} holds as well.
The exterior angle at a non-convex vertex of a polyhedron is a bit difficult to define.
The Gauss image of such a vertex is a non-convex and possibly self-overlapping spherical polygon, and the curvature is the properly defined area of this polygon, see \cite{Banchoff1967,Banchoff1970}.

\subsection{Theorema Egregium and the intrinsic Gauss--Bonnet theorem}
In the geometry of surfaces in $\R^3$ one can distinguish between \emph{extrinsic} and \emph{intrinsic} notions and quantities.
The intrinsic ones are solely determined by the lengths of the curves in the surface.
For example, the shortest curve between two points, and the angle between two curves are intrinsic.
Intrinsic quantities do not change when the surface is isometrically deformed.
The extrinsic notions refer to objects exterior to the surface.
For example, the Gauss map and the shape operator are extrinsic.
The more surprising is the fact that the Gaussian curvature, the determinant of the shape operator is intrinsic.

\begin{thm}[Theorema Egregium]
The Gaussian curvature $K = \det S = \kappa_1\kappa_2$ of a smooth surface in $\R^3$ is intrinsic.
\end{thm}

The proof of Theorema Egregium can be found in any textbook on differential geometry, for example in \cite{Spivak1979II}.

The formulas which express the Gaussian curvature in terms of the first fundamental form can be used to extend the definition of the Gaussian curvature to abstract surfaces equipped with Riemannian metrics.
The Gauss--Bonnet theorem can then be reproved purely intrinsically and thus generalized to a larger class of objects.

\begin{thm}[Gauss--Bonnet, intrinsic version]
\label{thm:GBIntr}
Let $S$ be a closed abstract smooth surface equipped with a Riemannian metric $g$, and let $K_g$ be the Gaussian curvature of this metric.
Then one has
\[
\int_S K_g\, d\area_g = 2\pi \chi(S).
\]
\end{thm}
Again, the proof can be found in textbooks, see e.g. \cite{Spivak1979III}.

Let us now state and prove the discrete analoga of Theorema Egregium and of the intrinsic version of Gauss-Bonnet.

\begin{thm}[Discrete Theorema Egregium]
The vertex curvatures of a polyhedral surface $M$ in $\R^3$ are intrinsic.
\end{thm}
\begin{proof}
Assume that $M$ is convex at a vertex $v$.
Then the normal cone to $M$ at $v$ is spanned by the normals to the faces of $M$ incident to $v$.
It follows that the faces of the normal cone are orthogonal to the edges of $M$ incident to~$v$.
It follows that the angle between two adjacent faces of the normal cone is complementary to the angle between the corresponding edges of $M$, see Figure \ref{fig:DThEg}, left.

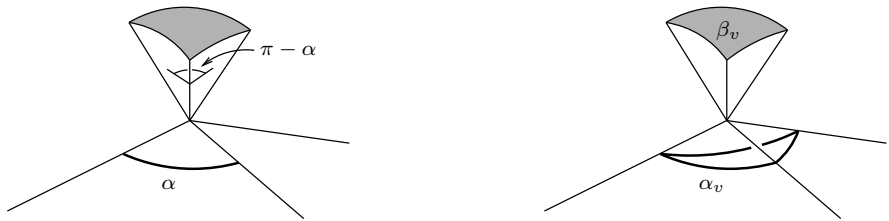
\begin{figure}[ht]
\begin{center}
\begin{picture}(0,0)%
\includegraphics{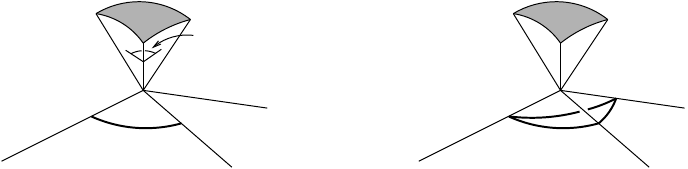}%
\end{picture}%
\setlength{\unitlength}{4144sp}%
\begingroup\makeatletter\ifx\SetFigFont\undefined%
\gdef\SetFigFont#1#2#3#4#5{%
  \reset@font\fontsize{#1}{#2pt}%
  \fontfamily{#3}\fontseries{#4}\fontshape{#5}%
  \selectfont}%
\fi\endgroup%
\begin{picture}(5229,1281)(2839,-388)
\put(3764,-199){\makebox(0,0)[lb]{\smash{{\SetFigFont{8}{9.6}{\sfdefault}{\mddefault}{\updefault}{\color[rgb]{0,0,0}$\alpha$}%
}}}}
\put(4348,599){\makebox(0,0)[lb]{\smash{{\SetFigFont{8}{9.6}{\sfdefault}{\mddefault}{\updefault}{\color[rgb]{0,0,0}$\pi-\alpha$}%
}}}}
\put(6944,-199){\makebox(0,0)[lb]{\smash{{\SetFigFont{8}{9.6}{\sfdefault}{\mddefault}{\updefault}{\color[rgb]{0,0,0}$\alpha_v$}%
}}}}
\put(7048,713){\makebox(0,0)[lb]{\smash{{\SetFigFont{8}{9.6}{\sfdefault}{\mddefault}{\updefault}{\color[rgb]{0,0,0}$\beta_v$}%
}}}}
\end{picture}%
\end{center}
\caption{The vertex curvature is intrinsic.}
\label{fig:DThEg}
\end{figure}

Thus the intersection of the normal cone with the unit sphere is a spherical polygon with the angles $\pi-\alpha_1, \ldots, \pi-\alpha_n$, where $\alpha_1, \ldots, \alpha_n$ are the angles at $v$ of faces incident to $v$.
By the formula for the area of a spherical polygon one has
\[
\beta_v = \sum_{i=1}^n (\pi - \alpha_i) - (n-2)\pi = 2\pi - \sum_{i=1}^n \alpha_i.
\]
But $\alpha_v := \sum_{i=1}^n \alpha_i$ is an intrinsic quantity, the length of the unit circle on the surface centered at $v$ (rescale if there are vertices of the surface at distance $< 1$), see Figure \ref{fig:DThEg}, right.

The argument can be modified to work for non-convex surfaces, see \cite{Banchoff1967,Banchoff1970}.
\end{proof}

The polyhedral analog of Theorema Egregium is mentioned already in the Hilbert--Cohn-Vossen classic book \cite{Hilbert1952} and goes probably even further back.

In order to state the intrinsic Gauss--Bonnet theorem we need a discrete analog of a surface with a Riemannian metric.
An \emph{abstract polyhedral surface} is a topological surface glued from Euclidean triangles by isometries between pairs of edges, see Figure \ref{fig:ConeMetr}.
On a polyhedral surface one can measure lengths and angles.
Every point has a neighborhood which is isometric either to an open subset of the Euclidean plane or to a neighborhood of the apex of a cone in $\R^3$.
The total angle $\alpha_v$ of the cone may be smaller or larger that $2\pi$.

\begin{figure}[ht]
\begin{center}
\begin{picture}(0,0)%
\includegraphics{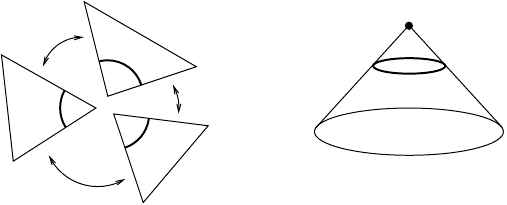}%
\end{picture}%
\setlength{\unitlength}{4144sp}%
\begingroup\makeatletter\ifx\SetFigFont\undefined%
\gdef\SetFigFont#1#2#3#4#5{%
  \reset@font\fontsize{#1}{#2pt}%
  \fontfamily{#3}\fontseries{#4}\fontshape{#5}%
  \selectfont}%
\fi\endgroup%
\begin{picture}(3845,1554)(394,-883)
\put(3383,-15){\makebox(0,0)[lb]{\smash{{\SetFigFont{8}{9.6}{\sfdefault}{\mddefault}{\updefault}{\color[rgb]{0,0,0}$\alpha_v$}%
}}}}
\end{picture}%
\end{center}
\caption{A fragment of an abstract polyhedral surface.}
\label{fig:ConeMetr}
\end{figure}
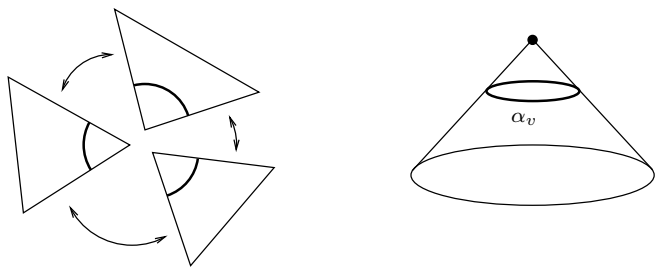

\begin{dfn}
The number
\[
K_v := 2\pi - \alpha_v
\]
is called the (intrinsic) \emph{curvature} of a vertex $v$ of an abstract polyhedral surface.
\end{dfn}
Any polyhedral surface in $\R^3$ has an underlying abstract polyhedral surface.
As shown in Theorem \ref{thm:GBIntr}, for every vertex the intrinsic curvature is equal to the extrinsic curvature.

\begin{thm}[Discrete Gauss--Bonnet, intrinsic version]
For every abstract polyhedral surface $M$ the sum of the intrinsic curvatures of its vertices is equal to $2\pi$ times the Euler characteristic:
\[
\sum_v K_v = 2\pi \chi(M).
\]
\end{thm}
\begin{proof}
Let $V, E, F$ be the number of vertices, edges, and faces of the polyhedral surface $M$.
One has
\[
\sum_v K_v = \sum_v (2\pi - \alpha_v) = 2\pi V - \sum_v \alpha_v.
\]
Observe that
\[
\sum_v \alpha_v = \pi F,
\]
since the sum of all $\alpha_v$ is the sum of all angles of all triangles of $S$, and the angle sum of a triangle is $\pi$.
It follows that
\[
\sum_v K_v = (2V - F)\pi.
\]
On the other hand,
\[
\chi(S) = V - E + F = V - \frac{F}2,
\]
because by counting the edges of all triangles one gets $2E = 3F$.
This implies the formula in the theorem.
\end{proof}

\begin{rem}
\label{rem:MetrStr}
Our definition of an abstract polyhedral surface contains an extra structure, that of a subdivision of the surface into triangles.
One and the same surface can be triangulated in many different ways.
Therefore one either has to introduce an equivalence relation (path-isometry between polyhedral surfaces) or to give a definition without extra data.
The latter can be done by defining a \emph{cone-surface} as a topological surface with an atlas taking values in open subsets of $\R^2$ and of Euclidean cones, whose transition maps are path-isometries.
\end{rem}

\subsection{Surfaces of bounded total curvature}
In the mid-20th century, A.~D.~Alexandrov and his school developed the notion of an abstract surface with bounded integral curvature, see \cite{Aleksandrov1967}.
These are topological surfaces with a compatible metric which is intrinsic (that is the distance between any two points is equal to the infimum of lengths of curves joining these two points) and can be uniformly approximated by smooth or polyhedral metrics of bounded total curvature.
The intrinsic metric on the boundary of a convex body is of this type.

Reshetnyak has proposed an alternative approach to surfaces of bounded total curvature, through the uniformization with a non-smooth conformal factor, see his book \cite{Reshetnyak1993} and a recent collection of translations and surveys \cite{Fillastre2022}.

\subsection{Higher dimensions}
\label{sec:HighDim}
It is not hard to generalize Steiner--Minkowski formula to convex bodies in $\R^n$:
\[
\area(M_r) = \sum_{k=0}^{n-1} V_k r^k,\quad
\vol(N_r) = \vol(N) + \sum_{k=1}^n \frac{V_{k-1}}k r^k,
\]
where
\[
V_k =
\begin{cases}
\int_M \sigma_k(\kappa)\, d\area, &\text{ if }M = \partial N\text{ is smooth},\\
\sum\limits_{\dim F = n-k-1} \beta_F\vol(F), &\text{ if }N\text{ is a polyhedron}.
\end{cases}
\]
Here $\sigma_k(\kappa) = \sum\kappa_{i_1}\cdots\kappa_{i_k}$ is the $k$-th elementary symmetric polynomial in the principal curvatures, and the sum in the discrete case ranges over all faces of $N$ of an appropriate dimension, $\beta_F$ being the exterior angle at the face $F$.

For every $k$, the coefficient $V_k$ is proportional to the average volume of projections of $N$ to $(n-k)$-dimensional linear subspaces of $\R^n$.

The coefficient $V_n$ is equal to the volume of the $n$-dimensional unit ball for the same reasons as in $\R^3$, so that one has
\[
\sum_v \beta_v = \omega_{n-1}, \quad \int_{\partial N} \kappa_1 \cdots \kappa_{n-1}\, d\area = \omega_{n-1},
\]
where $\omega_{n-1}$ is the area of the unit sphere $\Sph^{n-1}$.
The latter formula was generalized by Hopf as follows.
\begin{thm}
\label{thm:Hopf}
Let $N$ be a compact subset or $\R^n$ with smooth boundary.
Then one has
\[
\int_{\partial N} \kappa_1 \cdots \kappa_{n-1}\, d\area = \chi(N) \omega_{n-1}.
\]
\end{thm}
The integrand $\kappa_1 \cdots \kappa_{n-1}$ (known as the Gauss-Kronecker curvature) is intrinsic for the induced metric on the boundary $M = \partial N$ only if $n$ is odd, that is $\dim M$ is even.
The corresponding higher-dimensional generalization of the intrinsic Gauss--Bonnet theorem is called Gauss--Bonnet--Chern theorem, see Section \ref{sec:ChGB}.
More generally, among the coefficients of the Steiner--Minkowski formula those with odd indices are intrinsic (for example, $V_1$ is the area of $M$).

For a historic account on the Gauss-Bonnet-Hopf and the Gauss-Bonnet-Chern theorem see \cite{Gottlieb1996}.

\section{Curvature of Riemannian and polyhedral manifolds}
\label{sec:Riem}
\subsection{Riemannian manifolds and the total scalar curvature}
The curvature of a Riemannian manifold $M$ is described by the Riemann curvature tensor.
By contracting the Riemann tensor once, one obtains the Ricci tensor, by contracting once more the scalar curvature, which is a function on~$M$.

Another path to the scalar curvature is to first define the sectional curvature of a plane $V \subset T_pM$ as the intrinsic Gaussian curvature at $p$ of an as straight as possible subsurface of $M$ tangent to $V$, for example of the image of $V$ under the exponential map.
Then the scalar curvature at $p$ is a dimensional constant times the average sectional curvature of $V \subset T_pM$.

Denote the scalar curvature of the metric $g$ on $M$ by
\[
\scal_g \colon M \to \R.
\]
The \emph{total scalar curvature} of $(M,g)$ is the integral
\begin{equation}
\label{eqn:TotScal}
F(g) = \int_M \scal_g\, d\vol_g.
\end{equation}
If $\dim M =2$, then the scalar curvature is twice the Gaussian curvature, so that $F(g)$ is independent of the choice of $g$.
In higher dimensions one has the following variational formula.

\begin{thm}
Let $h$ be any symmetric $(0,2)$-tensor field on $M$.
Then the derivative of $F(g)$ in the direction of $h$ is equal to
\[
D_h F(g) = \int_M \left\langle \frac{\scal_g}2 g - \Ric_g, h \right\rangle\, d\vol_g,
\]
where $\Ric_g$ is the Ricci $(0,2)$-tensor of $g$.
\end{thm}

An almost immediate consequence of this is the following.
\begin{cor}
Critical points of the functional \eqref{eqn:TotScal} are Ricci-flat metrics (i.e. metrics with zero Ricci tensor) on $M$.
Critical points of the restriction of $F$ to the space of metrics of constant volume are Einstein metrics, that is metrics with
\[
\Ric_g = \lambda g \text{ for some }\lambda \in \R.
\]

If $\dim M = 3$, then critical points of $F$ are flat metrics (i.e. metrics locally isometric to $\R^3$) on $M$, and the critical points of the restriction to metrics of constant volume are metrics of constant sectional curvature.
\end{cor}

The functional \eqref{eqn:TotScal} is often called the Einstein-Hilbert functional.
For a comprehensive survey see \cite{Besse2008}.

\subsection{The total scalar curvature of polyhedral manifolds}
\label{sec:ScalPolyh}
An abstract polyhedral $n$-dimensional manifold is defined similarly to an abstract polyhedral surface.
One takes a finite set of Euclidean tetrahedra and identifies their faces isometrically pairwise so that the resulting topological space is an $n$-manifold.
Instead of tetrahedra one can take any other polyhedra, this does not add anything new as any polyhedron can be triangulated.

As in the case of polyhedral surfaces (see Remark \ref{rem:MetrStr}) it is the intrinsic metric of the resulting space which matters, not its decomposition into polyhedra.
From a metric point of view, a polyhedral manifold is Euclidean almost everywhere with the exception of the singular locus which is a subset of the union of codimension-$2$ faces of the constituting polyhedra.

The sum of the dihedral angles around each of these faces is called the \emph{cone angle} associated with this face, see Figure \ref{fig:ConeAngle}.

\begin{figure}[ht]
\includegraphics[width=.4\textwidth]{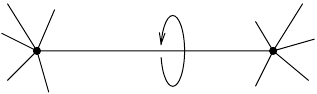}
\caption{A local view of the singular set of a $3$-dimensional Euclidean polyhedral manifold and a cone angle around one of the codimension-$2$ faces.}
\label{fig:ConeAngle}
\end{figure}

In dimension $3$ the singular set consists of vertices and edges, see Figure \ref{fig:ConeAngle}.
An edge does not need to connect two vertices, it can also be homeomorphic to a circle as the following example shows.

\begin{exl}
Take a cube and fold its faces along thick lines shown in Figure \ref{fig:BorromMan}, left.
(There are three more lines to fold along on the invisible faces of the cube.)
The result is a $3$-sphere whose singular set forms Borromean rings.
The cone angle at each of the components of the singular set is equal to $\pi$.
\end{exl}

\begin{figure}[ht]
\includegraphics[width=.2\textwidth]{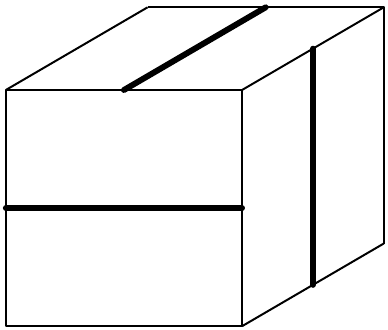}
\hspace{.03\textwidth}
\includegraphics[width=.2\textwidth]{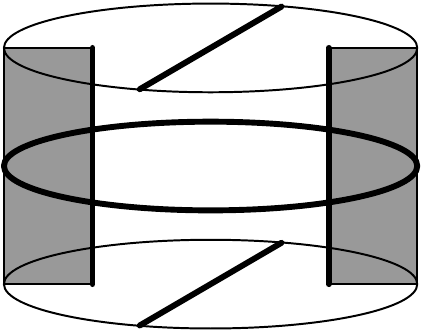}
\hspace{.03\textwidth}
\includegraphics[width=.2\textwidth]{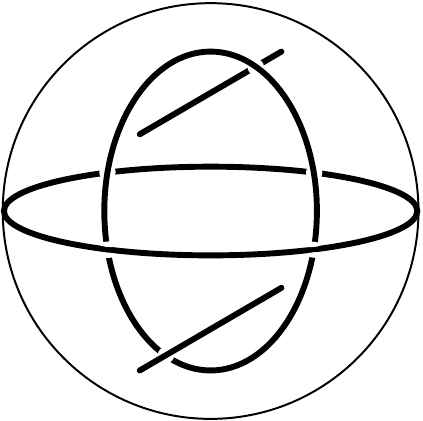}
\hspace{.03\textwidth}
\includegraphics[width=.2\textwidth]{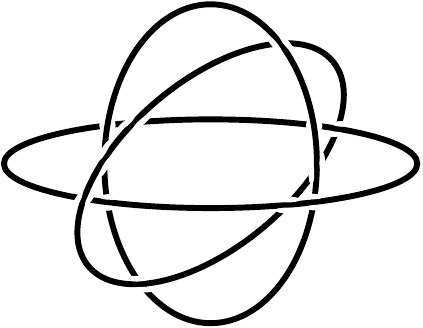}
\caption{Folding the faces of a cube produces a $3$-sphere with the Borromean rings as the singular set.}
\label{fig:BorromMan}
\end{figure}

\begin{dfn}
\label{dfn:DiscScalCurv}
The \emph{discrete total scalar curvature} of an $n$-dimensional manifold $M$ equipped with a polyhedral metric $c$ is defined as
\begin{equation}
\label{eqn:DiscrTotScal}
F(c) = a_n \sum_{Q} K_Q \vol(Q).
\end{equation}
Here $a_n$ is a dimensional constant, the sum ranges over all codimension-$2$ faces of the singular set, and $K_Q$ denotes the curvature of the face $Q$, that is $2\pi$ minus the cone angle around $Q$.
\end{dfn}

Note that formula \eqref{eqn:DiscrTotScal} produces the same value if it is applied to any subdivision of $M$ into Euclidean polyhedra.

For $3$-dimensional polyhedral manifolds one has
\[
F(c) = \sum_e K_e \ell_e,
\]
where the sum is taken over all edges, $\ell_e$ denotes the length, and $K_e$ the curvature of the edge $e$.

In order to deform a polyhedral metric, choose any triangulation of the manifold $(M,c)$.
Then any small change of the edge lengths of this triangulation results in a new polyhedral metric.

\begin{thm}
\label{thm:VarDiscr}
For any deformation of a polyhedral metric the following identity holds:
\begin{equation}
\label{eqn:VarDiscr}
dF(c) = a_n \sum_Q K_Q\ d\vol(Q)
\end{equation}
\end{thm}
\begin{proof}
The Schl\"afli formula for an $n$-dimensional Euclidean tetrahedron says that for any deformation of the tetrahedron one has
\[
\sum_Q \vol(Q) d\alpha_Q = 0,
\]
where $\alpha_Q$ is the dihedral angle at a codimension-$2$ face $Q$.
Applying this formula to each tetrahedron of a triangulation of $M$ and summing up one obtains
\[
\sum_Q \vol(Q) dK_Q = 0,
\]
which implies formula \eqref{eqn:VarDiscr}.
\end{proof}

\begin{cor}
A polyhedral metric $c$ is a critical point of the discrete total scalar curvature functional if and only if
\[
\sum_Q K_Q \frac{\partial \vol(Q)}{\partial \ell} = 0
\]
for all edges $\ell$ of any triangulation of $M$.

If $\dim M = 3$, then critical points are Euclidean metrics.
\end{cor}

The discrete total scalar curvature was introduced by a physicist Regge~\cite{Regge1961}.

\begin{rem}
More generally, Schl\"afli's formula for simplices in the $n$-dimensional space of constant sectional curvature $\kappa$ says
\[
\sum_Q \vol(Q) d\alpha_Q = (n-1)\kappa d\vol.
\]
This allows to extend the definition of the total scalar curvature to polyhedral manifolds glued from hyperbolic or spherical simplices in such a way that the variational formula \eqref{eqn:VarDiscr} still holds.
For example, in dimension $3$ one has
\[
F(c) = 2\kappa \vol_c(M) + \sum_e \ell_e K_e,
\]
where $\kappa \in \{1, -1\}$ is the curvature of the model space.
Observe that the summand $2\kappa \vol_c(M)$ is exactly the total scalar curvature of the regular part of $M$.
\end{rem}

\subsection{Lipschitz-Killing curvatures}
The scalar curvature is only the first in a series of curvature quantities.
The curvature tensor $R_g$ of a Riemannian metric $g$ can be turned into a self-adjoint operator on the exterior square of the tangent bundle:
\[
R_g \colon \Lambda^2 TM \to \Lambda^2 TM.
\]
Taking a tensor power and applying antisymmetrization one gets for every $k < n/2$ an operator
\[
R_g^k \colon \Lambda^{2k} TM \to \Lambda^{2k} TM.
\]
The trace of this operator is called the $k$-th \emph{Lipschitz-Killing curvature} of the metric $g$.
(Traces of tensor powers of an operator are the coefficients of the characteristic polynomial; in our situation the presence of antisymmetrization replaces the determinant with the Pfaffian.)
The total Lipschitz-Killing curvatures are the integrals
\[
S_{2k}(g) = \int_M \mathrm{tr} (R_g^k)\, d\vol_g.
\]
In particular, $S_2$ is the total scalar curvature \eqref{eqn:TotScal}.

The total Lipschitz-Killing curvatures appear in the following theorem known as the Weyl tube formula.
\begin{thm}[Weyl]
Let $M \subset \R^p$ be a closed smooth $n$-dimensional manifold.
Then for all sufficiently small $r > 0$ the volume of the $r$-neighborhood of $M$ is a polynomial in $r$ whose coefficients are proportional to the Lipschitz-Killing curvatures of $M$:
\[
\vol(M_r) = \sum_{k=0}^{\lfloor\frac{n}2\rfloor} a_{p,n,k} S_{2k}(g) r^{p-n+2k}.
\]
Here $a_{p,n,k} \in \R$ are dimensional constants, and $g$ denotes the induced metric on $M$.
In particular, the volume of the $r$-neighborhood is intrinsic, that is independent of an isometric embedding of $M$ in $\R^p$.
\end{thm}

In the codimension-1 case this is related to the Steiner formula, see Theorem \ref{thm:Steiner}.
Namely, the formula for the area of the parallel surface holds for $r<0$ provided $|r|$ is small enough.
Thus integrating $\area(M_x)$ for $x$ from $-r$ to $r$ one gets for $\dim M = 2$ the following:
\[
\vol(M_r) = 2r\, \area(M) + \frac{4\pi}3 r^3 \chi(M).
\]

For $n > 2$ integrating the formula for $\area(M_x)$ from Section \ref{sec:HighDim} one sees the symmetric polynomials of odd degrees in $\kappa_i$ disappear.
The polynomials of even degree are proportional to the Lipschitz--Killing curvatures due to the Gauss equation: the curvature operator of a hypersurface is the ``square'' of the shape operator.

The discrete Lipschitz--Killing curvatures are defined similarly to the discrete total scalar curvature, see Definition \ref{dfn:DiscScalCurv}.
\begin{dfn}
The \emph{discrete total $k$-th Lipschitz--Killing curvature} of an $n$-dimensional manifold $M$ equipped with a polyhedral metric $c$ is defined as
\[
S_{2k}(c) = a_{n,k} \sum_{\dim Q = n-2k} K_Q \vol(Q),
\]
where $K_Q$ is called the curvature of $M$ at a face $Q$ and is defined as
\begin{equation}
\label{eqn:LKDensity}
K_Q = \sum_{T \supset Q} (-1)^{\dim T - \dim Q} \beta(Q, T),
\end{equation}
where the sum is taken over all polyhedra $T$ of some polyhedral decomposition of $M$, and $\beta(Q, T)$ is the normalized external angle of $T$ at $Q$.
In particular, $\beta(Q, T) = \frac12$ if $\dim T - \dim Q = 1$.
One also puts $\beta(Q, Q) = 1$.
\end{dfn}
Observe the notation change resulting from the normalization: before now the total plane angle was $2\pi$, and the total solid angle $4\pi$.

In the case $k=1$, if $Q$ is surrounded by $m$ full-dimensional faces with internal angles at $Q$ equal to $\alpha_i$, then one has
\begin{equation}
\label{eqn:KEven}
K_Q = 1 - \frac{m}2 + \sum_{i=1}^m \left(\frac12 - \alpha_i\right) = 1 - \sum_{i=1}^m \alpha_i,
\end{equation}
which agrees with Definition \ref{dfn:DiscScalCurv} up to the normalization.

The following theorem was proved in \cite{Cheeger1984}, see also \cite{Lafontaine1987}.

\begin{thm}
\label{thm:CMS}
The curvature $K_Q$ in \eqref{eqn:LKDensity} is well-defined, that is independent of the polyhedral decomposition of $M$.

If a Riemannian metric $g$ on $M$ is approximated by a sequence of polyhedral metrics $c_n$ in a certain way, then $S_{2k}(c_n)$ converges to $S_{2k}(g)$.
\end{thm}

The well-definedness of $K_Q$ can be proved by giving an explicit formula in terms of the internal angles of faces incident with $Q$:
\[
K_Q = \sum_{m \ge 0} \sum_{T_m \supset \cdots \supset T_1 \supset Q} (-1)^m \alpha(Q, T_1) \alpha(T_1, T_2) \cdots \alpha(T_{m-1}, T_m).
\]
There is also a formula which uses only the relative angles of even codimension thus generalizing the formula \eqref{eqn:KEven}, see \cite{Cheeger1984,Budach1989}.

\subsection{Chern--Gauss--Bonnet theorem}
\label{sec:ChGB}
The Chern--Gauss--Bonnet theorem generalizes the intrinsic Gauss--Bonnet theorem from $2$-dimensional to all even-dimensional Riemannian manifolds.

\begin{thm}
\label{thm:ChGB}
Let $(M,g)$ be a Riemannian manifold of an even dimension~$n$.
Then its Euler characteristic is up to a dimensional factor equal to its total $\frac{n}2$-th Lipschitz--Killing curvature:
\[
\int_M \mathrm{tr}(R_g^{\frac{n}2})\, d\vol_g = a_n \chi(M).
\]
\end{thm}
Observe that the field of operators $R_g^{\frac{n}2}$ acts on a $1$-dimensional vector bundle $\Lambda^nTM$, so that its trace is simply the corresponding scaling factor.

Theorem \ref{thm:ChGB} was first proved in full generality by Chern \cite{Chern1944}.
In a preceding work of Allendoerfer and Weil \cite{Allendoerfer1943} the same formula was proved for manifolds isometrically embeddable in a Euclidean space, by an argument using the Weyl tube formula.
By Nash's isometric embedding theorem this poses no restriction, but Nash's theorem is hard and was proved only in 1956.

As noticed in \cite{Cheeger1984}, a discrete analog of the Chern--Gauss--Bonnet theorem holds and has a three line proof.

\begin{thm}
Let $M$ be a closed even-dimensional manifold equipped with a polyhedral metric.
Then one has
\[
\sum_v K_v = \chi(M),
\]
where the sum is over all vertices of $M$.
\end{thm}
\begin{proof}
By definition \eqref{eqn:LKDensity} of $K_v$ one has
\[
\sum_v K_v = \sum_v \sum_{T \ni v} (-1)^{\dim T} \beta(v,T) = \sum_T (-1)^{\dim T} \sum_{v \in T} \beta(v,T).
\]
But the sum of the normalized external angles at all vertices of a convex polyhedron is equal to $1$, thus denoting by $f_k$ the number of faces of $M$ of dimension $k$ one gets
\[
\sum_v K_v = \sum_T (-1)^{\dim T} = \sum_{k=0}^n (-1)^k f_k = \chi(M).
\]
\end{proof}

Together with the (admittedly difficult) approximation result from Theorem \ref{thm:CMS} this yields a new proof of the smooth Chern--Gauss--Bonnet theorem.

\end{document}